\documentclass[12pt]{amsart}
\usepackage{amscd,amssymb,verbatim}
\usepackage{amssymb,amsmath,amsthm,epsfig}

  \theoremstyle{plain}
  \newtheorem{thm}{Theorem}[section]
  \theoremstyle{definition}
  \newtheorem{defn}[thm]{Definition}
  \theoremstyle{plain}
  \newtheorem{lem}[thm]{Lemma}
  \theoremstyle{plain}
  \newtheorem{cor}[thm]{Corollary}
  \theoremstyle{plain}
  \newtheorem{prop}[thm]{Proposition}
  \theoremstyle{remark}
  \newtheorem*{claim*}{Claim}
  \newtheorem*{rem}{Remark}

\newcommand{\N}{\mathbb{N}}

\newcommand{\R}{\mathbb{R}}

\newcommand{\coo}{c_{00}}

\newcommand{\kcdot}{\!\cdot\!}
\newcommand{\keq}{\!=\!}

\newcommand{\kin}{\!\in\!}
\newcommand{\kge}{\!>\!}
\newcommand{\kgeq}{\!\geq\!}

\newcommand{\kle}{\!<\!}
\newcommand{\kleq}{\!\leq\!}

\newcommand{\kminus}{\!-\!}

\newcommand{\ksubset}{\!\subset\!}

\newcommand{\ktimes}{\!\times\!}

\newcommand{\vp}{\varepsilon}

\begin{document}

\title{Banach Spaces of Bounded Szlenk index II}
\author{D. Freeman}
\address{Department of Mathematics\\ Texas A\&M University\\
College Station, TX 77843, USA} \email{freeman@math.tamu.edu}
\author{E. Odell}
\address{Department of Mathematics \\
The University of Texas\\1 University Station C1200\\
Austin, TX 78712  USA} \email{odell@math.utexas.edu}
\author{Th. Schlumprecht}
\address{Department of Mathematics, Texas A\&M University\\
College Station, TX 77843, USA}
\email{thomas.schlumprecht@math.tamu.edu}

\author{A. Zs\'ak}
\address{School of Mathematics, University of Leeds, \\ Leeds, LS2 9JT,
  United Kingdom}
\email{zsak@maths.leeds.ac.uk}

\thanks{Research of the second and third author was supported by the
  National Science Foundation.
 This paper forms a portion of the doctoral dissertation of the first author which is being prepared
 at Texas A\&M University under the direction of the third author.}
\subjclass[2000]{46B20, 54H05}
\keywords{Szlenk index, universal space, embedding into FDDs,
  Effros-Borel structure, analytic classes}

\maketitle

\begin{abstract}
For every $\alpha<\omega_1$ we establish the existence of a separable Banach space whose Szlenk index is
 $\omega^{\alpha\omega+1}$ and which is universal for all separable Banach spaces
whose Szlenk-index does not exceed $\omega^{\alpha\omega}$. In order
to prove that result we provide an intrinsic characterization of
which Banach spaces embed into a space admitting an FDD with upper
 estimates.
\end{abstract}

\section{Introduction}\label{S1}
The added structure and rich theory of coordinate systems can be of
significant help when studying Banach spaces. Because of this, it is
often the case that Banach spaces are studied in the context of
being a subspace or quotient of some space with a coordinate system.
Two early results in this area are that every separable Banach space
is the quotient of $\ell_1$ and also every separable Banach space
may be embedded as a subspace of $C[0,1]$.  Both $\ell_1$ and
$C[0,1]$ have bases, and so we have in particular that every
separable Banach space is a quotient of a Banach space with a basis
and may also be embedded as a subspace of a Banach space with a
basis. However, it is often that one has a Banach space with a
particular property, and one wishes that the coordinate system has
some associated property. The first important step in this direction
was made by Zippin \cite{Z} who proved the following two major
results: every separable reflexive Banach space may be embedded as a
subspace of a space with shrinking and boundedly complete basis, and
every Banach space with separable dual may be embedded in a Banach
space with a shrinking basis. Further results in this area give
intrinsic characterizations on when a space may be embedded as a
subspace of a reflexive space with unconditional basis \cite{JZh},
or a reflexive space with an asymptotic $\ell_p$ FDD \cite{OSZ}. These
are only a portion of the recent results in this area. These new
characterizations are all based on the relatively recent tool of
weakly null trees. One important result in particular for us is a
characterization of subspaces of reflexive spaces with an FDD
satisfying subsequential $V$ upper block estimates and subsequential $U$
lower block estimates where $V$ is an unconditional, block stable, and
right dominant basic sequence and $U$ is an unconditional, block
stable, and left dominant basic sequence \cite{OSZ2}.  This
characterization when applied to Tsirelson spaces was shown to
have strong applications to the Szlenk index of reflexive spaces
\cite{OSZ3}. Our main result adds to this theory with the following
theorem which extends the results in \cite{OSZ2} and \cite{OSZ3} to
spaces with separable dual.
 The  notions and concepts used will be introduced in the next section.

\begin{thm}\label{T1}
Let $X^*$ be separable and $V=(v_i)$ be a normalized, 1-unconditional,
block stable, right dominant, and shrinking basic sequence. Then the
following are equivalent.
\begin{enumerate}
\item[1)]$X$ has subsequential $V$ upper tree estimates.
\item[2)]$X$ is a quotient of a space $Z$ with $Z^*$ separable and
$Z$ has subsequential $V$ upper tree estimates.
\item[3)]$X$ is a quotient of a space $Z$ with a shrinking  FDD satisfying
subsequential $V$ upper block estimates.
\item[4)]There exists a $w^*-w^*$ continuous embedding of $X^*$ into
$Z^*$, a space with boundedly complete FDD $(F_i^*)$ (so $Z=\oplus F_i$
defines $Z^*$) satisfying subsequential $V^*$ lower block estimates.
\item[5)]$X$ is isomorphic to a subspace of a space $Z$ with a shrinking FDD
satisfying subsequential $V$ upper  block estimates.
\end{enumerate}
\end{thm}

Using our characterization, we are able to achieve the following
universality result:

\begin{thm}\label{T2}
Let $V=(v_i)$ be a normalized, 1-unconditional, shrinking, block
stable, and right dominant basic sequence. There is a Banach space $Z$
with a shrinking FDD $(F_i)$ satisfying subsequential $V$ upper block
estimates such that if a Banach space $X$ with separable dual
satisfies subsequential $V$ upper tree estimates, then $X$ embeds into
$Z$.
\end{thm}

We will apply Theorems~\ref{T1} and~\ref{T2} for the case that $V$ is
the canonical basis of $T_{\alpha,c}$, the Tsirelson space of
order~$\alpha$ and parameter~$c$, which will allow us to prove some
new results for the Szlenk index. As shown in \cite{OSZ3}, the Szlenk
index is closely related to a space having subsequential
$T_{\alpha,c}$ upper tree estimates for some $0<c<1$. In
particular, for each $\alpha<\omega_1$ if a Banach space $X$ with
separable dual has Szlenk index at most $\omega^{\alpha\omega}$, then
$X$ satisfies subsequential $T_{\alpha,c}$ upper tree
estimates for some $c\in(0,1)$. In~\cite{OSZ3} the converse is also
shown in the case that $X$ is reflexive. Our characterization allows
us to extend this to the class of spaces with separable dual. We give
the following theorem.
\begin{thm}\label{T3} Let $\alpha<\omega_1$.  For a space $X$ with separable
dual, the following are equivalent:
\begin{enumerate}
\item[(i)]$X$ has Szlenk index at most $\omega^{\alpha\omega}$.
\item[(ii)]$X$ satisfies subsequential $T_{\alpha,c}$ upper tree estimates for
some $c\in(0,1)$.
\item[(iii)]$X$ embeds into a space $Z$ with an FDD $(E_i)$ which
satisfies subsequential $T_{\alpha,c}$ upper block estimates in $Z$ for
some $c\in(0,1)$.
\end{enumerate}
\end{thm}

Note that the implication (iii)$\Rightarrow$(i) shows that the space
$Z$ in~(iii) also has Szlenk index at most
$\omega^{\alpha\omega}$. In particular, since the unit vector
basis of $T_{\alpha,c}$ satisfies subsequential $T_{\alpha,c}$ upper
block estimates, the same is true for $T_{\alpha,c}$. The above
structure theorem then says that the Tsirelson spaces $T_{\alpha,c}$
form a sort of upper envelope for the class of spaces with separable
dual and with Szlenk index at most $\omega^{\alpha\omega}$.

We are able to combine the previous two theorems using ideas in
\cite{OSZ3} to prove the following universality result.
\begin{thm}\label{T4}
For each $\alpha<\omega_1$ there exists a Banach space $Z$ with a
shrinking FDD and Szlenk index at most $\omega^{\alpha\omega+1}$
such that $Z$ is universal for the collection of spaces with
separable dual and Szlenk index at most $\omega^{\alpha\omega}$.
\end{thm}
In particular, the universal space $Z$ will be of the form
$(\sum_{n\in\N}X_n)_{\ell_2}$, where $X_n$ has an FDD satisfying
subsequential $T_{\alpha,\frac{n}{n+1}}$ upper block estimates and $X_n$ is
universal for all Banach spaces with separable dual which satisfy
subsequential  $T_{\alpha,\frac{n}{n+1}}$ upper tree estimates.

Theorem \ref{T4} represents a quantitative version of a result first
shown by Dodos and Ferenczi \cite{DF}, which states that for every
$\alpha<\omega_1$ there is a Banach space with separable dual which is
universal for all separable Banach spaces whose Szlenk index does not
exceed $\alpha$. As well as finding a bound on the Szlenk index of
this universal space, we also express, as mentioned above, the
topological property of having a certain Szlenk index in terms of norm
estimates in which the Tsirelson spaces play an essential
r\^ole. While the proofs in \cite{DF} use methods of descriptive set
theory  developed by Bossard~\cite{B}, our proofs will rely on
concepts like {\em infinite asymptotic games}, trees and branches as
introduced in \cite{OS1} and \cite{OS3}.

\section{Definitions and lemmas}\label{S2}

Our main result characterizes subspaces and quotients of spaces
having a shrinking FDD with subsequential $V$ upper block estimate, where $V$
is an unconditional, right dominant, block stable, and shrinking basic
sequence. The case when $V=T_{\alpha,c}$ is a Tsirelson space is
intimately related to the Szlenk index,
and has become an important property in the fertile area between
descriptive set theory and the classification of Banach spaces
\cite{OSZ3}. For $\alpha<\omega_1$ and $c\in(0,1)$, the definition of
the Tsirelson space $T_{\alpha,c}$ of order $\alpha$ and parameter
$c$, and the relevant properties of $T_{\alpha,c}$ for us can also be
found in \cite{OSZ3}.

 For basic notions like (shrinking and boundedly
complete) FDDs
 and their projection constants and blockings we refer to
  \cite{OSZ2}.
 If $Z$ is a Banach spaces with an FDD $E=(E_i)$,
 we denote by $c_{00}(\oplus E_i)$ the dense linear subspace of $Z$ spanned by $(E_i)$ and its closure by $[E_i]=[E_i]_Z$.
  We denote the closure of $c_{00}(\oplus E^*_i)$ inside $Z^{*}$ by $Z^{(*)}$. If
 $(E_i)$ is shrinking it follows that $Z^{(*)}=Z^*$ and if $(E_i)$ is
 boundedly complete, then $Z^{(*)}$ is the predual of $Z$.
  If $A\subset\N$ is finite, or cofinite, we denote
  the natural projection onto the closed span of $(E_i:i\in A)$ by $P_A^E$, i.e.
  $$P_A^E: Z\to Z,\quad  P_A^E\Big(\sum_{i=1}^\infty x_i\Big)=
  \sum_{i\in A} x_i, \text{ whenever }x_i\in E_i \text{ for $i\in \N$
    so that }\sum_{i=1}^\infty x_i\in Z.$$
Let us also recall the following notion from \cite{OSZ2}.

\begin{defn}
Let $Z$ be a Banach space with an FDD $(E_n)$, let $V=(v_i)$ be a
normalized 1-unconditional basis, and let $1\leq C<\infty$. We say
that $(E_n)$ {\em satisfies subsequential }
$C$-$V$-{\em upper block estimates} if every normalized block sequence
$(z_i)$ of $(E_n)$ in $Z$ is $C$-dominated by $(v_{m_i})$, where
$m_i=\min\textrm{supp}_E (z_i)$ for all $i\in\N$.
 We say that $(E_n)$ {\em satisfies subsequential }
$C$-$V$-{\em lower block estimates} if every normalized block sequence
$(z_i)$ of $(E_n)$ in $Z$ $C$-dominates $(v_{m_i})$, where
$m_i=\min\textrm{supp}_E (z_i)$ for all $i\in\N$.
 We say that $(E_n)$ {\em satisfies subsequential}
  $V$-{\em upper (or lower) block
 estimates} if it satisfies subsequential $C$-$V
 $-upper (or lower) block estimates for some $1\leq C<\infty$.
\end{defn}
Subsequential $V$-upper block estimates and subsequential $V$-lower block
estimates are dual properties, as shown in the following proposition
from \cite{OSZ2}.
\begin{prop}\label{P0}{\rm \cite[Proposition 2.14]{OSZ2}}
Assume that $Z$ has an FDD $(E_i)$, and let $V=(v_i)$ be a
1-unconditional normalized basic sequence with biorthogonal
functionals $V^*=(v_i^*)$.  The following statements are equivalent:
\begin{enumerate}
\item[(a)]$(E_i)$ satisfies subsequential $V$-upper block estimates in
$Z$.\\
\item[(b)]$(E^*_i)$ satisfies subsequential $V^*$-lower block estimates in
$Z^{(*)}$.\\
\end{enumerate}
Moreover, if $(E_i)$ is bimonotone in $Z$, then the equivalence
holds true if one replaces, for some $C\geq1$, $V$-upper estimates
by $C$-$V$-upper estimates in $(a)$ and $V^*$-lower block estimates by
$C$-$V^*$-lower block estimates in $(b)$.
\end{prop}

It is important to note that if a Banach space $Z$ has an FDD
$(E_n)$ which satisfies subsequential $V$ upper block estimates where
$V=(v_i)$ is weakly null, then $(E_n)$ is shrinking.  Indeed, any
normalized block sequence of $(E_n)$ is dominated by a weakly null
sequence, and is thus weakly null.  Thus if $V$ is weakly null, a
necessary condition for a Banach space $X$ to be isomorphic to a
quotient or subspace of a Banach space with an FDD satisfying
subsequential $V$-upper block estimates is that $X$ have separable dual.
This is important as spaces with separable dual may be analyzed
using {\em weakly null trees} . In this paper we will
  need in particular {\em weakly null even trees} (see \cite{OSZ2}).

In order to index weakly null even trees, we denote
$$T_\infty^{\textrm{even}}=\{(n_1,n_2,...,n_{2\ell})\,:\,\,n_1<n_2<...<n_{2\ell}\textrm{ are in }\N\textrm{ and }\ell\in\N\}.$$

\begin{defn}
If $X$ is a Banach space, an indexed family  $(x_\alpha)_{\alpha\in
T_\infty^{\textrm{even}}}\subset X$ is called an {\em even
tree}. Sequences of the form
$(x_{n_1,...,n_{2\ell-1},k})_{k=n_{2\ell-1}+1}^\infty$ are called
nodes. Sequences of the form
$(n_{2\ell-1},x_{n_1,...,n_{2\ell}})_{\ell=1}^\infty$ are called
branches. A normalized tree, i.e.~one with $||x_{\alpha}||=1$ for all $\alpha\in
T_\infty^{\textrm{even}}$, is called {\em weakly null} if
every node is a weakly null sequence. 

If $T'\subset T_\infty^{\textrm{even}}$  is
closed under taking restrictions so that for each $\alpha\in
T'\cup\{\emptyset\}$ and for each $m\in\N$ the set
$\{n\in\N:\,(\alpha,m,n)\in T'\}$ is either empty or has infinite
size, and moreover the latter occurs for infinitely many values of
$m$, then we call $(x_\alpha)_{\alpha\in T'}$ a \emph{full subtree of
}$(x_\alpha)_{\alpha\in T_\infty^{\textrm{even}}}$. Note that
$(x_\alpha)_{\alpha\in T'}$ could then be relabeled to a family
indexed by $T_\infty^{\textrm{even}}$, and note that the branches of
$(x_\alpha)_{\alpha\in T'}$ are branches of
$(x_\alpha)_{\alpha\in T_\infty^{\textrm{even}}}$ and that the nodes of
$(x_\alpha)_{\alpha\in T'}$ are subsequences of certain nodes of
$(x_\alpha)_{\alpha\in T_\infty^{\textrm{even}}}$.

In case that $X$ has an FDD
$(E_i)$, we say that
  a normalized tree is a {\em block tree} (with respect to $(E_i)$) if
 every node is a block sequence with respect to   $(E_i)$.
Note that every weakly null tree has a full subtree which is a
perturbation of a block tree, and that if $(E_i)$ is shrinking, then
every block tree is weakly null.
\end{defn}

If $Z$ is a Banach space with an FDD $(E_n)$, and $X$ is a closed
subspace of $Z$ then any weakly null even tree in $X$ has a branch
equivalent to a block basis of $(E_n)$. Thus if $(E_n)$ satisfies
subsequential $V$-upper block estimates, then every weakly null even
tree in $X$ has a branch
$(n_{2\ell-1},x_{n_1,...,n_{2\ell}})_{\ell=1}^\infty$ such that
$(x_{n_1,...,n_{2\ell}})_{\ell=1}^\infty$ is dominated by
$(v_{n_{2\ell-1}})_{\ell=1}^\infty$.  We make this into a coordinate
free condition with the following definition.

\begin{defn}
Let $X$ be a Banach space, $V=(v_i)$ be a normalized 1-unconditional
basis, and $1\leq C<\infty$. We say that $X$ {\em satisfies
subsequential } $C$-$V$-{\em upper tree estimates} if every weakly
null even tree $(x_\alpha)_{\alpha\in T_\infty^{\textrm{even}}}$ in
$X$ has a branch
$(n_{2\ell-1},x_{n_1,...,n_{2\ell}})_{\ell=1}^\infty$ such that
$(x_{n_1,...,n_{2\ell}})_{\ell=1}^\infty$ is $C$-dominated by
$(v_{n_{2\ell-1}})_{\ell=1}^\infty$.

 We say that $X$ {\em satisfies subsequential }$V$-{\em upper
 tree estimates} if it satisfies subsequential $C$-$V$-upper tree estimates for
 some $1\leq C<\infty$.

If $X$ is a subspace of a dual space, we say that $X$
 {\em satisfies subsequential } $C$-$V$-{\em lower
}$w^*${\em tree estimates} if every $w^*$ null even tree
$(x_\alpha)_{\alpha\in T_\infty^{\textrm{even}}}$ in $X$ has a
branch $(n_{2\ell-1},x_{n_1,...,n_{2\ell}})_{\ell=1}^\infty$ such
that $(x_{n_1,...,n_{2\ell}})_{\ell=1}^\infty$ $C$-dominates
$(v_{n_{2\ell-1}})_{\ell=1}^\infty$.

\end{defn}





We have a property of trees and a property of FDDs, and our goal is
to show how they are related.  Zippin's theorem allows us to embed a
Banach space with separable dual into a space with shrinking FDD.
Our next step will be to then pass information about trees in the
space to information about $\bar{\delta}$-skipped blocks of the FDD,
which we define here.

\begin{defn}
Let $\textbf{E}=(E_i)$ be an FDD for a Banach space $Y$ and let
$\bar{\delta}=(\delta_i)$ with $\delta_i\downarrow0$.  A sequence
$(y_i)\subset S_Y$ is called a $\bar{\delta}-\textsl{skipped block
w.r.t. } (E_i)$ if there exist integers $1=k_0<k_1<...$ so that for
all $i\in\N$,
$$\Vert P^E_{(k_{i-1},k_i)}y_i-y_i\Vert<\delta_i.$$
\end{defn}

 The following proposition is an adaptation
of Proposition~2.18 in \cite{OSZ2} for the case $(E_i)$ is shrinking,
 but not necessarily boundedly complete and for the case where $X$ is a $w^*$-closed subspace of a dual space. We will need to  first
 recall some notation introduced in \cite{OSZ2}.

Given a Banach space $X$, we let $(\N\!\times\! S_X)^\omega$ denote the
set of all sequences $(k_i,x_i)$, where $k_1<k_2<\dots$ are
positive integers, and $(x_i)$ is a sequence in $S_X$. We equip the
set $(\N\!\times\! S_X)^\omega$ with the product topology of the
discrete topologies of $\N$ and $S_X$. Given $ A\subset (\N\ktimes
S_X)^\omega$ and $\vp\kge 0$, we let
\[
 A_\vp = \Big\{ \big((\ell_i,y_i):i\in\N\big)\kin (\N\ktimes S_X)^\omega :\,\exists\,
\big((k_i,x_i):i\in\N\big)\kin A\quad k_i\kleq\ell_i\,,\ \|x_i\kminus
  y_i\|\kle\vp\kcdot 2^{-i}\ \forall\, i\kin\N\Big\}\ ,
\]
and we let $\overline{ A}$ be the closure of $ A$ in
$(\N\ktimes S_X)^\omega$.

Given $ A\ksubset (\N\ktimes S_X)^\omega$, we say that an even tree
$(x_\alpha)_{\alpha\in   T_\infty^{\text{even}}}$ in $X$ \emph{has a branch in $ A$ }if
there exist $n_1\kle n_2\kle\dots$ in $\N$ such that
$\big((n_{2i-1},x_{(n_1,n_2,\dots,n_{2i})}):i\in\N\big)\kin A$.

Let $Z$ be a Banach space with an FDD $(E_i)$ and assume that $Z$
contains $X$. Let $C$ be the projection constant of $(E_i)$ in $Z$.
For each $m\kin\N$ we set $Z_m\keq\overline{\bigoplus _{i>m}
    E_i}$. Given $\vp\kge 0$, we consider the following game between
  players S (subspace chooser) and P (point chooser). The game has an
  infinite sequence of moves; on the $n^{\text{th}}$ move ($n\kin\N$)
  S picks $k_n,m_n\kin\N$ and P responds by picking $x_n\kin S_X$
    with $d(x_n,Z_{m_n})\kle\vp'\kcdot 2^{-n}$, where
    $\vp'\keq\min\{\vp,1\}$. S wins the game if the sequence
    $(k_i,x_i)$ the players generate ends up in
    $\overline{A_{5C\vp}}$, otherwise P is declared the winner. We
    will refer to this as the $(A,\vp)$-game. Note that the definition
    of $S$ winning is slightly different from the one given
    in~\cite{OSZ2}. This is because of the extra complication of
    dealing with the non-reflexive case.

\begin{prop}\label{P1}
Let $X$ be an infinite-dimensional closed subspace of a space $Z$
with an FDD $(E_i)$. Let $A\subset(\N\times S_X)^\omega$. If $(E_i)$
is shrinking, or if $Z$ is a dual space with $(E_i)$ boundedly
complete and $X$ $w^*$ closed in $Z$, then the following are equivalent.
\begin{enumerate}
\item[(a)] For all $\vp>0$ there exists $(K_i)\subset\N$ with
$K_1<K_2<...$, $\overline{\delta}=(\delta_i)\subset(0,1)$ with
$\delta_i\searrow0$, and a blocking $F=(F_i)$ of $(E_i)$ such that
if $(x_i)\subset S_X$ is a $\overline{\delta}$-skipped block
sequence of $(F_n)$ in $Z$ with
$||x_i-P^F_{(r_{i-1},r_i)}x_i||<\delta_i$ for all $i\in\N$, where
$1\leq r_0<r_1<r_2<...$, then
$(K_{r_{i-1}},x_i)\in\overline{A_\vp}$.
\item[(b)]For all $\vp>0$ $S$ has a winning strategy for the
$(A,\vp)$-game.
\end{enumerate}
If $(E_i)$ is shrinking, then (a) and (b) are equivalent to
\begin{enumerate}
\item[(c)] for all $\vp>0$ every normalized, weakly null even
tree in $X$ has a branch in $\overline{A_\vp}$.
\end{enumerate}
If $Z$ is a dual space with $(E_i)$ boundedly
complete and $X$ $w^*$ closed in $Z$, then (a) and (b) are equivalent
to
\begin{enumerate}
\item[(d)] for all $\vp>0$ every normalized, $w^*$ null even
tree in $X$ has a branch in $\overline{A_\vp}$.

\end{enumerate}
\end{prop}
\begin{proof}
The proofs of the implications
$(b)\Rightarrow(a)\Rightarrow(d)\Rightarrow(b)$ shown in the
reflexive case \cite[Proposition 2.18]{OSZ2} still hold in the
nonreflexive case when $Z$ is a dual space with $(E_i)$ boundedly
complete and $X$ $w^*$ closed in $Z$: we use $w^*$ compactness of
$B_X$, and the fact that $(E_i)$ is biorthogonal to a shrinking FDD of
a predual of $Z$ (instead of weak compactness of $B_X$ and the
shrinking property of $(E_i)$ as in~\cite{OSZ2}).

For the case in which $(E_i)$ is shrinking, the proofs of the
implications $(b)\Rightarrow(a)\Rightarrow(c)$  shown for the
reflexive case still work. The proof for the implication
$(c)\Rightarrow(b)$ requires some adaptation which we provide here.

We start with a preliminary result: let
$Z_m=\overline{\bigoplus_{i>m}E_i}$ as before, and let $C$ be the
projection constant of $(E_i)$ in $Z$. Then for every $\eta>0$ with
$(1+C)\eta<1$ and for every sequence $(z_i)\subset S_X$ with
$d(z_i,Z_i)<\eta$ for all
$i\in\N$, there is a subsequence $(x_i)$ of $(z_i)$ and a weakly null
sequence $(y_i)\subset S_X$ such that $||x_i-y_i||<2(1+C)\eta$ for all
$i\in\N$. Indeed, we can pass to a weakly Cauchy subsequence $(x_i)$ of
$(z_i)$ such that
\begin{equation}\label{P1E1}
||P^E_{[1,n]}(x_i-x_j)||< \eta 2^{-j}\quad\forall n\in\N\textrm{
and }i>j\geq n.
\end{equation}
Since $d(x_i,Z_i)<\eta$, we have $||P^E_{[1,i]}x_i||<C\eta$ for all
$i\in\N$. Since $(E_i)$ is shrinking, the sequence
$(P^E_{(i,\infty)}x_i)$ is weakly null, and so for each $n\in\N$ there
exists $(a_i^{(n)})_{i=n}^{K(n)}=(a_i)_{i=n}^K\subset [0,1]$ such that
$\sum_{i=n}^K a_i=1$ and $||\sum_{i=n}^K a_i
P^E_{(i,\infty)}x_i||<\eta$. Set
$$y_n=\frac{x_n-\sum a_i x_i}{||x_n-\sum a_i x_i||}.$$
We have that
\begin{align*}
\left\Vert\sum_{i=n}^K a_i x_i\right\Vert\leq\sum_{i=n}^K
a_i||P^E_{[1,i]}x_i||+\left\Vert\sum_{i=n}^K a_i
P^E_{(i,\infty)}x_i\right\Vert <\sum_{i=n}^K a_i C\eta+\eta
=(1+C)\eta\ ,
\end{align*}
which implies that $||x_n-y_n||<2(1+C)\eta$. We also have
\begin{align*}
||P_{[1,n]}y_n||\leq 2\Big\|P_{[1,n]}(x_n-\sum a_i x_i)\Big\|
\leq2\sum a_i||P_{[1,n]}(x_n-x_i)||<2\eta 2^{-n}\ ,
\end{align*}
which tends to zero as $n\rightarrow\infty$. Hence $(y_n)\subset S_X$
is weakly null as $(E_i)$ is shrinking.

We now continue with the proof of the implication
$(c)\Rightarrow(b)$. Assume that $S$ does not have a winning strategy
for the $(A,\vp)$ game for some $\vp>0$.
As this game is determined \cite{Ma} , there exists a winning
strategy $\phi$ for the point chooser.  The function $\phi$ takes
values in $S_X$:   if $(k_i),(m_i)\in[\N]^\omega$  are the choices
by player $S$ and  if $z_n=\phi(k_1,m_1,...,k_n,m_n)$ for all
$n\in\N$, then $d(z_i,Z_{m_i})<\vp 2^{-i}$ for all $i\in\N$ and
$(k_i,z_i)\not\in \overline{A_{5C\vp}}$. We can, of course, assume that
$(1+C)\vp<1$. For each $\alpha\in T_\infty^{\textrm{even}}$ set
$z_\alpha=\phi(\alpha)$. Then $(z_\alpha)_{\alpha\in
  T_\infty^{\textrm{even}}}$ is a normalized even tree in $X$ no
branch of which is in $\overline{A_{5C\vp}}$. Its nodes
$(z_i)=(z_{(k_1,\dots,k_{2\ell-1},i)})_{i>k_{2\ell-1}}$ satisfy
$d(z_i,Z_i)<\vp 2^{-\ell}$ for all $i>k_{2\ell-1}$. By repeated
applications of our preliminary observation we can find a full subtree
$(x_\alpha)_{\alpha\in T_\infty^{\textrm{even}}}$ of
$(z_\alpha)_{\alpha\in T_\infty^{\textrm{even}}}$ and a weakly
null even tree $(y_\alpha)_{\alpha\in T_\infty^{\textrm{even}}}$ such
that $||x_\alpha-y_\alpha||<2(1+C)\vp
2^{-\ell}$ for all $\ell\in\N$ and $\alpha=(k_1,\dots,k_{2\ell})\in
T_\infty^{\textrm{even}}$. Since no branch of $(x_\alpha)_{\alpha\in
  T_\infty^{\textrm{even}}}$ is in
$\overline{A_{5C\vp}}$, it follows that no branch of
$(y_\alpha)_{\alpha\in T_\infty^{\textrm{even}}}$ is in
$\overline{A_{C\vp}}$. Thus $(c)$ fails.
\end{proof}
\begin{rem}
We will be applying Proposition~\ref{P1} for the case
$A=\{(n_i,x_i)_{i=1}^\infty \,|  (v_{n_i}) \textrm{ dominates  }
(x_i)\}$ where $(v_i)$ is a 1-unconditional basic sequence. We will
also be repeatedly applying the technique of blocking FDDs. For this
reason it is important that properties of $\bar{\delta}$-skipped
blocks are preserved by blockings. This follows at once from the
following simple observation: Assume that $(E_i)$ is an FDD with
projection constant $K$, and $(H_i)$ is a blocking of $(E_i)$. Then a
$\bar{\delta}$-skipped block of $(H_i)$ is a $2K\bar{\delta}$-skipped
block of $(E_i)$.
\end{rem}

We will be concerned with a space $X$ which satisfies subsequential
$V$-upper tree estimates.  However the nature of our proofs require
us to work with $X^*$ as well. This is because some of the blocking
techniques which we use depend on the FDD being boundedly
complete. Before stating a duality result on upper tree estimates, we
need the following definition: A basic sequence $V=(v_i)$ is {\em $C$-right dominant }(respectively, \emph{$C$-left-dominant}) if for all
sequences $m_1\kle m_2\kle \dots$ and $n_1\kle n_2\kle\dots$ of
positive integers with $m_i\kleq n_i$ for all $i\kin\N$ we have that
$(v_{m_i})$ is $C$-dominated by (respectively, $C$-dominates)
$(v_{n_i})$. We say that $(v_i)$ is \emph{right-dominant }or
\emph{left-dominant }if for some $C\kgeq 1$ it is $C$-right-dominant
or $C$-left-dominant, respectively.

\begin{lem}\label{L3}
Let $X$ be a Banach space with separable dual, and let $V=(v_i)$ be a
normalized, 1-unconditional, right dominant basis.  If $X$ satisfies
subsequential $V$-upper tree estimates, then $X^*$ satisfies
subsequential $V^*$-lower $w^*$ tree estimates.
\end{lem}
\begin{proof}
$X$ has separable dual, so by  \cite[Corollary 8]{DFJP} there exists
a space $Z$ with a shrinking and bimonotone FDD $(F_i)$ for which
there is a quotient map $Q:Z\rightarrow X$.  After renorming $X$ we
may assume that it has the quotient norm $||x||=\inf_{Qy=x}||y||$
for all $x\in X$. This gives that $Q^*$ is an isometric embedding of
$X^*$ into $Z^*$. Furthermore, $(F^*_i)$ is a boundedly complete FDD
for $Z^*$ as $(F_i)$ is shrinking.

We next show that if $(x_i)\subset S_{X^*}$ is $w^*$ null, then there
is a subsequence $(x'_i)$ of $(x_i)$ and a weakly null sequence
$(y_i)\subset S_X$ such that $x'_i(y_i)>\frac34$ for all $i\in\N$.

We have 
$(Q^*x_i)$ is $w^*$ null in $Z^*$ as $Q^*$ is $w^*$ to $w^*$
continuous.  Hence there is a subsequence $(x'_i)$ of $(x_i)$ such
that
$||P^{F^*}_{[1,i)}Q^*x'_i||<\frac14$ for all $i$. As $(F_i)$ is bimonotone,
there exists $(z_i)\subset S_Z$ such that
$||P^{F}_{[1,i)}z_i||=0$ and
$Q^*x'_i(z_i)>\frac34$. Note that $||Q(z_i)||>\frac34$ for all $i$,
and that the sequence
$(z_i)$ is coordinate-wise null and
hence weakly null as $(F_i)$ is shrinking. It follows that
$y_i=\frac{Q(z_i)}{||Q(z_i)||}$ defines a weakly null sequence in
$S_X$ with $x'_i(y_i)>\frac34$ for all $i$.

Now let $(x_\alpha)_{\alpha\in T^{\textrm{even}}_\infty}\subset
S_{X^*}$ be a $w^*$ null tree. By repeated applications of the above
result, there is a full subtree $(x'_\alpha)$ of $(x_\alpha)$ and a
weakly null tree $(y_\alpha)$ in $X$ such that
$x'_\alpha(y_\alpha)>\frac34$ for all $\alpha\in
T^{\textrm{even}}_\infty$. Passing to further subtrees if necessary,
we can also assume that for $k<\ell$ in $\N$ and for
$\alpha=(n_1,\dots,n_{2k}),\ \beta=(n_1,\dots,n_{2\ell})$ in
$T^{\textrm{even}}_\infty$ we have $\max\{
|x'_\alpha(y_\beta)|,|x'_\beta(y_\alpha)|\}<4^{-\ell}$.

Let $(n_{2k-1},y_{(n_1,...,n_{2k})})_{k=1}^\infty$ be a branch of
the weakly null tree $(y_\alpha)_{\alpha\in T^{\textrm{even}}_\infty}$
such that $(v_{n_{2k-1}})_{k=1}^\infty$ $C$-dominates
$(y_{(n_1,...,n_{2k})})_{k=1}^\infty$ for some $C\geq 1$.
Let $(a_i)\in\coo$ such that $||\sum a_i v^*_i||=1$. There exists
$(b_i)\in\coo$ such that $||\sum b_i v_i||=1$ and $\sum a_i b_i=1$
as $(v_i)$ is bimonotone. Furthermore,
$\textrm{sign}(a_i)=\textrm{sign}(b_i)$ as $(v_i)$ is
1-unconditional. We have that,
\begin{align*}
1&=||\sum_{i=1}^\infty a_i v_i^*||\\&
 =\sum_{i=1}^\infty a_i b_i\\
&\leq\frac43\sum_{i=1}^\infty a_i b_i x'_{(n_1,...,n_{2i})}(y_{(n_1,...,n_{2i})})\\
&\leq\frac{4}{3}(\sum_{k=1}^\infty a_k
x'_{(n_1,...,n_{2k})})(\sum_{\ell=1}^\infty b_\ell
y_{(n_1,...,n_{2\ell})})+\frac{4}{3}\sum_{k=1}^\infty\sum_{\ell\neq
k}|x'_{(n_1,...,n_{2k})}(y_{(n_1,...,n_{2\ell})})|\\
&\leq \frac{4}{3}(\sum_{k=1}^\infty a_k
x'_{(n_1,...,n_{2k})})(\sum_{\ell}^\infty
b_\ell y_{(n_1,...,n_{2\ell})})+\frac{2}{3}\\
&< C\frac{4}{3}||\sum_{k=1}^\infty a_k
x'_{(n_1,...,n_{2k})}||+\frac{2}{3}.
\end{align*}
Hence $(x'_{(n_1,...,n_{2k})})_{k=1}^\infty$ $4C$-dominates
$(v^*_{n_{2k-1}})_{k=1}^\infty$. Finally, the branch
$(n_{2k-1},x'_{(n_1,...,n_{2k})})$ corresponds to a branch
$(m_{2k-1},x_{(m_1,...,m_{2k})})$ in the original tree with $n_i\leq
m_i$ for all $i\in\N$. Since $(v_i)$ is right dominant, $(v^*_i)$ is
left dominant, and hence $(x_{(m_1,...,m_{2k})})$ dominates
$(v^*_{m_{2k-1}})$. Thus $X^*$ satisfies subsequential $V^*$-lower
$w^*$ tree estimates.
\end{proof}

Proposition \ref{P1} allows us to pass from information about trees
to information about $\bar{\delta}$-skipped blocks of an FDD
$(E_n)$. To go from information about $\bar{\delta}$-skipped blocks
to blocks in general, we will renorm the FDD $(E_n)$ to form a new
space.

Let $Z$ be a space with an FDD $E=(E_n)$ and let $V=(v_i)$ be a
normalized 1-unconditional basic sequence.  The space $Z^V=Z^V(E)$
is defined to be the completion of $\coo(\bigoplus E_n)$ with
respect to the following norm $\Vert\cdot\Vert_{Z^V}$.
$$
\Vert z\Vert_{Z^V}=\max_{k\in\N,\,1\leq
n_0<n_1<...<n_k} \Big\Vert\sum_{j=1}^k\Vert P^E_{[n_{j-1},n_j)}(z)\Vert_Z
\cdot v_{n_{j-1}}\Big\Vert_V\,\,\,\textrm{ for }z\in\coo(E_i).
$$
We note that if $\Vert\cdot\Vert$ and $\Vert\cdot\Vert'$ are
equivalent norms on $Z$ then the corresponding norms
$\Vert\cdot\Vert_{Z^V}$ and $\Vert\cdot\Vert'_{Z^V}$ are equivalent
on $\coo(\bigoplus E_n)$.  This allows us, when examining the space
$Z^V$, to assume that $(E_n)$ is bimonotone in $Z$.  The following
proposition from \cite{OSZ2} is what makes the space $Z_V$ essential
for us. Recall that a basic sequence is called {\em  $C$-block stable}
 for some $C\ge 1$
  if any two normalized
  block bases $(x_i)$ and $(y_i)$ with
  \[
  \max \big( \text{supp} (x_i)\cup\text{supp}(y_i)\big) < \min \big( \text{supp}
  (x_{i+1})\cup\text{supp}(y_{i+1})\big)\qquad\text{for all }i\kin\N
  \]
  are $C$-equivalent. We say that $(v_i)$ is \emph{block-stable }if it
  is $C$-block-stable for some constant~$C$.

The following Proposition recalls some properties of $Z^V$ which were shown
 in \cite{OSZ2}.
\begin{prop}\label{P2}{\rm \cite[Corollary 3.2, Lemma 3.3 and 3.5]{OSZ2}}
Let $V=(v_i)$ be a normalized, 1-unconditional, and C-block-stable
basic sequence.  If $Z$ is a Banach space with an FDD $(E_i)$, then
$(E_i)$ satisfies $2C$-$V$-lower block estimates in $Z^V(E)$.

 If the basis $(v_i)$ is boundedly complete then $(E_i)$ is a boundedly complete
  FDD for $Z^V(E)$.

  If the basis $(v_i)$ is shrinking and if $(E_i)$ is shrinking in $Z$, then
   $(E_i)$ is a shrinking FDD for $Z^V(E)$.

\end{prop}

In proving our main theorem we will show that if $X$ satisfies
subsequential $V$-upper tree estimates then it is isomorphic to a
subspace of some $Z^V(E)$ and is isomorphic to a quotient of some
$Z^V(F)$.

\section{Proofs of the main results}\label{S3}

\begin{proof}[Proof of Theorem \ref{T1}]

\noindent$1)\Rightarrow 4)$ $(v_i)$ is $D$-right-dominant for some
$D\geq1$, from which we can easily deduce that $(v_i^*)$ is $D$-left-dominant.  By
\cite[Corollary 8]{DFJP} there exists a space $Z$ with a shrinking
and bimonotone FDD $(E_i)$ for which there is a quotient map
$Q:Z\rightarrow X$. The map $Q^*:X^*\rightarrow Z^*$ is an into
isomorphism. After renorming $X$ if necessary, we can assume that $X$
has the quotient norm induced by $Q$, and so $Q^*$ is an isometric embedding. By Lemma \ref{L3}
 we have that $X^*$ satisfies subsequential C-$V^*$ lower $w^*$ tree
 estimates for some $C\geq1$.  As $Q^*X^*\subset Z^*$ is $w^*$ closed, we may apply Proposition
 \ref{P1} with $A=\{(n_i,x_i)_{i=1}^\infty\in(\N\times S_{Q^*X^*})^\omega \,|  (x_i) \textrm{ C-dominates  }
(v_{n_i})\}$ and $\vp>0$ such that
$\overline{A_\vp}\subset\{(n_i,x_i)_{i=1}^\infty\in(\N\times S_{Q^*X^*})^\omega \,|  (x_i) \textrm{
2CD-dominates } (v_{n_i})\}$.  This gives sequences
$(K_i)\in[\N]^\omega$ and $\bar{\delta}=(\delta_i)\subset (0,1)$ and
a
 blocking $(F_i)$ of $(E^*_i)$ such that if $(x_i)\subset S_{Q^*X^*}$ and
  $||x_i-P^F_{(r_{i-1},r_i)}(x_i)||<2\delta_i$
 for some $(r_i)\in[\N]^\omega$ then $(K_{r_{i-1}},x_i)\in \overline{A_\vp}$. Hence, the sequence $(x_i)$ $2CD$-dominates
 $(v_{K_{r_{i-1}}})$.

We choose a blocking $G=(G_i)$ of $(F_i)$ defined by
$G_i=\sum_{j=m_{i-1}+1}^{m_i}F_j$ for some $(m_i)\in[\N]^\omega$ such
that there exists $(e_n)\subset S_{Q^*X^*}$ with $||e_n-P_n^G
(e_n)||<\delta_n/2$ for all $n\in\N$.

In order to continue we need a result from \cite{OS1}
 which is based on an argument due to W.~B.~Johnson \cite{J}.
\cite[Corollary 4.4]{OS1} was stated for reflexive spaces. Here we
state it for $w^*$-closed subspaces of dual spaces with a
boundedly complete FDD: the proof is easily seen to work in this
situation. Also note that conditions (d) and (e) which were not stated
in~\cite{OS1} follow easily from the proof.

\begin{prop}\label{P3}{\rm \cite[Lemma 4.3 and Corollary 4.4]{OS1}}
Let $Y$ be a $w^*$ closed subspace of a dual Banach space $Z$
with a boundedly complete FDD $A=(A_i)$ having projection constant
$K$.  Let $\bar{\eta}=(\eta_i)\subset(0,1)$ with
$\eta_i\downarrow0$. Then there exists
$(N_i)_{i=1}^\infty\in[\N]^\omega$ such that the following holds.
Given $(k_i)_{i=0}^\infty\in[\N]^\omega$ and $x\in S_Y$, there
exists $x_i\in Y$ and $t_i\in(N_{k_{i-1}-1},N_{k_{i-1}})$ for all
$i\in\N$ with $N_0=0$ and $t_0=0$ such that
\begin{enumerate}
\item[(a)] $x=\sum_{i=1}^\infty x_i, \textrm{and for all
}i\in\N\textrm{ we have}$\\
\item[(b)] $\textrm{either }||x_i||<\eta_i\textrm{ or
}||x_i-P^A_{(t_{i-1},t_i)}x_i||<\eta_i||x_i||$,\\
\item[(c)]$||x_i-P^A_{(t_{i-1},t_i)}x||<\eta_i$,\\
\item[(d)]$||x_i||<K+1$,\\
\item[(e)]$||P_{t_i}^A x||<\eta_i$.\\
\end{enumerate}
\end{prop}
We apply Proposition \ref{P3} with $Y=Q^*X^*$, $A=G$ and
$\bar{\eta}=\bar{\delta}$ which gives a sequence
$(N_i)\in[\N]^\omega$.  We set
$H_j=\bigoplus_{i=N_{j-1}+1}^{N_j}G_i$ for each $j\in\N$. To make
notation easier we let $V^*_M=(v_{M_i}^*)$ be the subsequence of
$(v^*_i)$ defined by $M_i=K_{m_{N_i}}$.

Fix $x\in S_{Q*X^*}$ and a sequence
$(n_i)_{i=0}^\infty\in[\N]^\omega$, the proof in \cite[Theorem 4.1(a)]{OSZ2}  shows
$$\Big\Vert\sum_{i=1}^\infty\Vert P^H_{[n_{i-1},n_i)}(x)\Vert_{Z^*}\cdot
v^*_{M_{n_{i-1}}}\Big\Vert_{V^*}\leq4D^2 C(1+2\Delta+2)+2+3\Delta.
$$
where $\Delta=\sum_{i=1}^\infty\delta_i$. Thus the norms
$||\cdot||_{Z^*}$ and $||\cdot||_{(Z^*)^{V^*_M}}$ are equivalent on
$Q^*X^*$.  As the norm on each $H_j$ is unchanged, a coordinate-wise
null sequence in $Q^*X^*\subset Z^*$ will still be coordinate-wise
null in $(Z^*)^{V^*_M}$.  Hence the map $Q^*:X^*\rightarrow
(Z^*)^{V^*_M}$ is still $w^*$ to $w^*$ continuous.

We have that $(Z^*)^{V^*_M}$ has a boundedly complete FDD $(H_j)$ which satisfies
subsequential $V^*_M$ lower block estimates by Proposition \ref{P2}.
We can now fill in the FDD as in~\cite[Lemma 2.13]{OSZ2}.  We let $B_{M_j}=H_j$ for all $j\in\N$
and we let $B_j=\R$ for each $j\not\in (M_i)$.  For
$x=(x_j)\in\coo(B_j)$ we define
$$\|x\|=\Big\|\sum_{j\in\N}x_{M_j}\Big\|_{(Z^*)^{V^*_M}}+\sum_{j\not\in
M}|x_j|.$$
We let $Y$ be the completion of $c_{00}(\oplus B_j)$ under this norm. $Y$ is
 clearly isometrically isomorphic to $(Z^*)^{V^*_M}\oplus \ell_1$ or
 $(Z^*)^{V^*_M}\oplus
 \ell^n_1$ for some $n\in\N_0$.
Thus the natural embedding of $(Z^*)^{V^*_M}$ into $Y$ is $w^*$ to
$w^*$ continuous.  Hence there is a $w^*$ to $w^*$ continuous
embedding of $X^*$ into $Y$. Finally, from
the fact that $(H_j)$ satisfies subsequential $V^*_M$ lower block estimates in
 $(Z^*)^{V^*_M}$ it is not hard to deduce that $(B_j)$ satisfies subsequential $V^*$
lower block estimates in $Y$.

 $4)\Rightarrow 3)$ This is clear
because if $(F_i^*)$ is a boundedly complete FDD of $Z^*$ then
   $(F_i)$ is a shrinking FDD of $Z$ and a $w^*-w^*$
continuous embedding $T:X^*\rightarrow Z^*$ must be the dual of
some quotient map $Q:Z\rightarrow X$.  Also, $(F_i^*)$  having
subsequential $V^*$ lower block estimates is equivalent to $(F_i)$
having subsequential $V$ upper block estimates due to Proposition
\ref{P0}.

$3)\Rightarrow 1)$ Let $(F_i)$ be a bimonotone shrinking FDD which
satisfies subsequential $V$ upper block estimates, and $Q:Z\rightarrow
X$ be a quotient map.  There exists $C>0$ such that $B_X\subset
Q(CB_Z)$. We will need the following lemma.

\begin{lem}\label{L1} Let $X$ and $Z$ be Banach spaces, $F=(F_i)$ be a bimonotone FDD for $Z$, and $Q:Z\rightarrow X$ be a quotient map. If $(x_i)\subset
S_X$ is weakly null and $Q(CB_Z)\supseteq B_X$ for some $C>0$ then
for all $\vp>0$ and $n\in\N$ there exists $N\in\N$ and $z\in 2CB_Z$
such that $P_{[1,n]} z=0$ and $|| Qz-x_N||<\vp$.
\end{lem}
\begin{proof}
Let $z_i\in C B_Z$ such that $Q z_i=x_i$.  After passing to a
subsequence $(z_{k_i})$ and perturbing we may assume instead that
$P_{[1,n]}^F {z_{k_i}}=z_0$ for some $z_0\in C B_Z$, and that $||Q
z_{k_i} -x_{k_i}||<\vp/3$.  As $(x_{k_i})$ is weakly null, 0 must be
in the closure of the convex hull of $(x_{k_i})$. Hence there is
some finite sequence $(a_i)_{i=2}^m\subset[0,1]$ such that
$||\sum_{i=2}^m a_i x_{k_i}||<\vp/3$ and $\sum_{i=2}^m a_i=1$.  Let
$z=z_{k_1}-\sum_{i=2}^m a_i z_{k_i}$. Then $z\in 2C B_Z$, $P_{[1,n]}^F
z=0$, and
$$||Qz-x_{k_1}||=||Qz_{k_1}-x_{k_1}-\sum_{i=2}^m a_i Qz_{k_i}||\leq||Qz_{k_1}-x_{k_1}||+\sum_{i=2}^m a_i || Q
z_{k_i}-x_{k_i}||+||\sum_{i=2}^m a_i x_{k_i}||< \vp.
$$
\end{proof}

\noindent{\em Continuation of proof of Theorem \ref{T1}.}
Let $(x_t)_{t\in T^{\textrm{even}}_\infty}\subset S_X$ be a weakly null even
tree in $X$, and let $\eta\in(0,1)$. By Lemma \ref{L1} we may pass to
a full subtree $(x'_t)_{t\in T^{\textrm{even}}_\infty}$ of $(x_t)$ so that
there exists a block tree $(z_t)_{t\in T^{\textrm{even}}_\infty}\subset 2C
B_Z$ such that $||Q(z_t)-x'_t||<\eta2^{-\ell}$ for all
$\ell\in\N$ and $t=(k_1,\dots,k_{2\ell})\in T^{\textrm{even}}_\infty$.
Now choose $1=k_1<k_2<\dots$ such that $\max
\text{supp}(z_{(k_1,...,k_{2i})})<k_{2i+1}<
\min \text{supp}(z_{(k_1,...,k_{2i+2})})$ for all $i\in\N$.  Then
$(z_{(k_1,...,k_{2i})})$ is dominated by $(v_{k_{2i-1}})$, and hence
$(x'_{(k_1,...,k_{2i})})$ is dominated by $(v_{k_{2i-1}})$ provided $\eta$
was chosen sufficiently small. Finally, the branch
$(k_{2i-1},x'_{(k_1,\dots,k_{2i})})$ corresponds to a branch
$(\ell_{2i-1},x_{(\ell_1,\dots,\ell_{2i})})$ in the original tree with
$k_i\leq \ell_i$ for all $i\in\N$. Since $(v_i)$ is right-dominant, it
follows that $(x_{(\ell_1,...,\ell_{2i})})$ is dominated by
$(v_{\ell_{2i-1}})$. Thus $X$ satisfies subsequential $V$ upper tree
estimates.

$2)\Rightarrow 1)$ We assume that $X$ is a quotient of a space $Z$
with separable dual such that $Z$ satisfies subsequential $V$ upper
tree estimates.  By the implication $1)\Rightarrow3)$ applied to $Z$,
$Z$ is the quotient of a space $Y$ with a shrinking FDD satisfying
subsequential $V$ upper block estimates.  $X$ is then also a quotient
of $Y$, so by the implication $3)\Rightarrow1)$ we have that $X$
satisfies subsequential $V$ upper tree estimates.

$1)\Rightarrow 5)$ Our proof will be based on the proof of~\cite[Theorem
4.1(b)]{OSZ2}. By Zippins theorem we may assume, after renorming
$X$ if necessary, that there exists a Banach space $Z$ with a
shrinking, bimonotone FDD $(F_j)$ and an isometric embedding
$i:X\hookrightarrow Z$. Also, by~\cite[Corollary 8]{DFJP} we know that
there exists a Banach space $W$ with a shrinking FDD $(E_j)$ and a
quotient map $Q:W\twoheadrightarrow X$. Thus we have a quotient map
$i^*:[F^*_j]=Z^*\twoheadrightarrow X^*$ and an embedding
$Q^*:X^*\hookrightarrow [E^*_j]=W^*$. We can assume, after renorming
$W$ if necessary, that $Q^*$ is an isometric embedding. Note that
$(F^*_j)$ and $(E^*_j)$ are boundedly complete FDDs of $Z^*$ and
$W^*$, respectively, and that $X^*$ has the quotient norm induced by
$i^*$. Let $K$ be the projection constant of $(E_j)$ in $W$.

By Lemma~\ref{L3}, $X^*$ satisfies subsequential $C$-$V^*$ lower
$w^*$ tree estimates for some $C\geq 1$. Choose $D\geq 1$ such that
$(v_i)$ is $D$ right dominant.
Since $Q^*X^*$ is $w^*$ closed in $W^*$, we can apply
Proposition~\ref{P1} as in the proof of the implication $1)\Rightarrow
4)$: after blocking $(E^*_j)$, we find sequences
$(K_i)\in[\N]^\omega$ and $\bar{\delta}=(\delta_i)\subset (0,1)$ with
$\delta_i\downarrow 0$ such that if $(x_i)\subset S_{Q^*X^*}$ is a
$2K\bar{\delta}$-skipped block of $(E^*_j)$ with
$||x_i-P^{E^*}_{(r_{i-1},r_i)}x_i||<2K\delta_i$ for all $i\in\N$,
where $1\leq r_0< r_1< r_2<\dots$, then $(v^*_{K_{r_{i-1}}})$ is
$2CD$-dominated by $(x_i)$, and moreover, using standard perturbation
arguments and making $\bar{\delta}$ smaller if necessary, we can
assume that if $(w_i)\subset W^*$ satisfies $||x_i- w_i||< \delta_i$
for all $i\in\N$, then $(w_i)$ is a basic sequence equivalent to
$(x_i)$ with projection constant at most $2K$. We can also assume that
$\Delta=\sum_{i=1}^{\infty}\delta_i<\frac17$.

Choose a sequence $(\vp_i)\subset (0,1)$ with $\vp_i\downarrow 0$ and
$3K(K+1)\sum_{j=i}^{\infty}\vp_j<\delta_i^2$ for all $i\in\N$. After
blocking $(E^*_j)$ if necessary, we can assume that for any subsequent
blocking $D$ of $E^*$ there is a sequence $(e_i)$ in $S_{Q^*X^*}$ such
that $||e_i-P^D_i(e_i)||<\vp_i/2K$ for all $i\in\N$.

Using Johnson and Zippin's blocking lemma \cite{JZ1} we may
assume, after further blocking our FDDs $(F^*_j)$ and $(E^*_j)$ if
necessary, that given $k<\ell$, if
$z^*\in\bigoplus_{j\in(k,\ell)}F_j^*$, then $\Vert
P^{E^*}_{[1,k)}Q^*i^*z^*\Vert<\vp_k$ and $\Vert
P^{E^*}_{[\ell,\infty)}Q^*i^*z^*\Vert<\vp_\ell$, and moreover the same
holds if one passes to any blocking of $(F^*_j)$ and the corresponding
blocking of $(E^*_j)$. Note that although the conditions of the
Johnson-Zippin lemma are not satisfied here, the proof is easily seen
to apply because our FDDs are boundedly complete, and the map $Q^*i^*$
is $w^*$ to $w^*$ continuous.

We now continue as in the proof of~\cite[Theorem 4.1(b)]{OSZ2}: we
replace $F^*_j$ by the quotient space $\tilde{F}_j=i^*(F^*_j)$, we let
$\tilde{Z}$ be the completion of $\coo(\tilde{F}_j)$ w.r.t.~the norm
$|||\cdot|||$ as defined in~\cite{OSZ2} and obtain a quotient map
$\tilde{\iota}\colon \tilde{Z}\to X^*$. We note that the results
corresponding to~\cite[Proposition 4.9(b),(c)]{OSZ2} are valid here as
their proof does not require reflexivity (part (a) is not required,
and indeed neither valid, here).

Finally, we find a blocking $(\tilde{G}_j)$
of $(\tilde{F}_j)$ and a subsequence $V^*_N=(v^*_{n_i})$ of $(v^*_i)$
such that $\tilde{\iota}$ is still a quotient map of
$\tilde{Z}^{V^*_N}(\tilde{G})$ onto $X^*$ and it is still $w^*$ to
$w^*$ continuous (note that $(\tilde{G}_j)$ is boundedly complete in
$\tilde{Z}^{V^*_N}(\tilde{G})$ by Proposition~\ref{P2}). Since
$(\tilde{G}_j)$ satisfies subsequential $V^*_N$ lower block estimates
in $\tilde{Z}^{V^*_N}(\tilde{G})$, statement $(5)$ will then follow by
duality (after filling the FDD as in the proof of the implication
$1)\Rightarrow 4)$). To find suitable $\tilde{G}$ and $(n_i)$ we now
follow the proof of~\cite[Theorem 4.1(b)]{OSZ2} verbatim. The only
comment we need to make is that~\cite[Lemma 4.10]{OSZ2} is valid
since we are working with boundedly complete FDDs and $w^*$ to $w^*$
continuous maps.

Finally, since the missing implications (5)$\Rightarrow$(1) and
(3)$\Rightarrow$(2) are trivial, we finished the proof of the
theorem.
\end{proof}
The proof  of the following result is an adaptation of the proof
Theorem~5.1 in \cite{OSZ2} to the non reflexive case.

\begin{cor}\label{C1}
Let $V=(v_i)$ be a 1-unconditional,  shrinking, block stable, and
right dominant
normalized basic sequence. There is a Banach space $Y$ with a shrinking FDD
$(E_i)$ satisfying subsequential $V$ upper  block estimates such that if a
Banach space $X$ with separable dual has subsequential $V$ upper tree
estimates, then $X$ embeds into $Y$.
\end{cor}
\begin{proof} By Schechtman's result \cite{S} there exists a space $W$
 with a bimonotone FDD $E=(E_i)$ with the property that any space $X$
 with bimonotone FDD $F=(F_i)$ naturally almost isometrically embeds
 into $W$, i.e. for any $\vp>0$ there is a $(1+\vp)$-embedding
 $T:X\to W$ and a sequence $(k_i)\in [\N]^\omega$, such that
 $T(F_i)=E_{k_i}$, and moreover $\sum_i P^E_{k_i}$ is a norm-1
 projection of $W$.

Since  $V^*$  is boundedly complete it follows from Proposition \ref{P2}  that
 the sequence $(E_i^*)$ is a boundedly complete FDD of the
space $(W^{(*)})^{V^*}$. It follows that $(E_i)$ is a shrinking FDD of the space
$Y=\big((W^{(*)})^{V^*}\big)^{(*)}$ and that $Y^*=(W^{(*)})^{V^*}$.
We denote by $\|\cdot\|_W$,    $\|\cdot\|_{W^{(*)}}$,  $\|\cdot\|_Y$,  $\|\cdot\|_{Y^*}$
the norms in $W$, $W^{(*)}$, $Y$ and $Y^*$, respectively.

By Proposition \ref{P2}  $(E_i^*)$ satisfies subsequential $V^*$
lower block estimates in  $(W^{(*)})^{V^*}$, and, thus,
 by Proposition \ref{P0}  $(E_i)$ satisfies subsequential $V$ upper block estimates  in $Y$ (recall
 that $Y^{(*)}=Y^*=(W^{(*)})^{V^*}$).

We now have to show that a space $X$ with separable dual and with subsequential $V$ upper tree estimates
 embeds in $Y$. By Theorem \ref{T1} we can assume that $X$ has a
 shrinking, bimonotone FDD $(F_i)$
 satisfying  subsequential  $V$ upper block estimates. By our choice of $W$ we can assume that
$X$ is the complemented  subspace of $W$  generated by a subsequence $(E_{k_i})$ of $(E_i)$.
We need to show that on $X$ the norms $\|\cdot\|_W$ and $\|\cdot\|_Y$ are equivalent.

Let $C\ge 1$  be chosen so that $(v_i)$ is  $C$-block stable and $C$-right dominant
  (thus
 $(v^*_i)$ is  $C$-block stable and $C$-left dominant) and such that $(E_{k_i}^*)$  satisfies
subsequential $C$-$V^*$ lower  block estimates in $X^*$.
Let $w^*\in c_{00}(\oplus F_i^*)= c_{00}(\oplus E_{k_i}^*)$. Clearly, we have  $\|w^*\|_{W^{(*)}}\le \| w^*\|_{Y^*}$.
Choose $1\le m_0<m_1<\ldots $ such that
$$\|w^*\|_{Y^*}=\Big\|\sum_{i=1}^\infty \|P^{E^*}_{[m_{i-1},m_i)}(w^*) \|_{W^{(*)}}v^*_{m_{i-1}}\Big\|_{V^*}.$$
W.l.o.g we can assume that $m_0=1$ and that $P^{E^*}_{[m_{i-1},m_i)}(w^*)\not=0$, for $i\in\N$.
Since $w^*\in c_{00}(\oplus E_{k_i}^*)$, we can choose $j_1<j_2<\ldots$ such that
 $k_{j_i}=\min\text{supp} P_{[m_{i-1},m_i)}^{E^*}(w^*)$ and deduce
\begin{align*}
 \|w^*\|_{Y^*}&=\Big\|\sum_{i=1}^\infty
 \|P^{E^*}_{[m_{i-1},m_i)}(w^*) \|_{W^{(*)}}  v^*_{m_{i-1}}\Big\|_{V^*}\\
&\le C\Big\|\sum_{i=1}^\infty \|P^{E^*}_{[m_{i-1},m_i)}(w^*) \|_{W^{(*)}}v^*_{k_{j_i}}\Big\|_{V^*}\\
&\le C^2\Big\|\sum_{i=1}^\infty \|P^{F^*}_{[j_{i},j_{i+1})}(w^*) \|_{W^{(*)}}v^*_{{j_i}}\Big\|_{V^*}
\le C^3 \|w^*\|_{W^{(*)}}.\\
  \end{align*}
This proves that $\|\cdot\|_{W^{(*)}}$ and $\|\cdot\|_{Y^{*}}$ are equivalent on
   $c_{00}(\oplus E^*_{k_i})$. Since $X$ is
1-complemented in $W$,  and $X^{*}$ is 1-complemented in $W^{(*)}$ and since $\sum_i P^{E^*}_{k_i}$ is still
 a  norm-1 projection from $Y^*$ onto
    $\overline{c_{00}(\oplus (E_{k_i}))}^{Y^*}$ it follows for any
    $w\in c_{00}(\oplus E_{k_i})$ that
$\frac1{C^3} \|w\|_W\le \|w\|_Y\le \|w\|_W$, which finishes the proof of our claim.
\end{proof}

As an application of Theorem \ref{T1}  we extend structural and
universality results on classes of bounded Szlenk index from the
reflexive case studied in~\cite{OSZ3} to the non-reflexive case.

\begin{cor}\label{C2}
 Let $\alpha<\omega_1$.  For a space $X$ with separable
dual, the following are equivalent:
\begin{enumerate}
\item[(i)]X has Szlenk index at most $\omega^{\alpha\omega}$,
\item[(ii)]X satisfies subsequential $T_{\alpha,c}$ upper tree estimates for
some $c\in(0,1)$,
\item[(iii)]X embeds into a space $Z$ with an FDD $(E_i)$ which
satisfies subsequential $T_{\alpha,c}$ upper block estimates in $Z$ for
some $c\in(0,1)$.
\end{enumerate}
\end{cor}

\begin{proof}
  The implication (i)$\Rightarrow$(ii) is proved in Corollary 19 and
  Theorem 21 of~\cite{OSZ3} (the reflexivity assumption there is not
  used for the relevant implication). The implication
  (iii)$\Rightarrow$(i) follows from~\cite[Propositon
  17]{OSZ3}. Finally, (ii)$\Rightarrow$(iii) follows from the
  implication (1)$\Rightarrow$(5) of Theorem~\ref{T1}.
\end{proof}

\begin{cor}\label{C3}
For each $\alpha<\omega_1$ there exists a Banach space $Z_\alpha$
with a shrinking FDD and Szlenk index at most
$\omega^{\alpha\omega+1}$ such that $Z_\alpha$ is universal for the
collection of spaces with separable dual and Szlenk index at most
$\omega^{\alpha\omega}$.
\end{cor}

\begin{proof}
By Corollary \ref{C1} for all $n\in\N$ there exists a Banach space
$X_n$ with an FDD satisfying subsequential
$T_{\alpha,\frac{n}{n+1}}$ upper block estimates which is universal for
all Banach spaces with separable dual which satisfy subsequential
$T_{\alpha,\frac{n}{n+1}}$ upper tree estimates. Let
$Z_\alpha=(\bigoplus X_n)_{\ell_2}$.  We have that $Z_\alpha$ is
universal for the collection of spaces with separable dual and
Szlenk index at most $\omega^{\alpha\omega}$ by Corollary \ref{C2}.
The Szlenk index of $Z_\alpha$ is at most $\omega^{\alpha\omega+1}$
as proven in \cite{OSZ3}.
\end{proof}

\end{document}